\documentclass{article}

\usepackage{graphicx,float}
\usepackage{algorithm}
\usepackage{algpseudocode}
\algrenewcommand\algorithmicforall{\textbf{for all}}
\algnewcommand{\FOR}[1]{\State\algorithmicfor\ #1\ \algorithmicdo}
\algnewcommand{\ENDFOR}{\algorithmicend\ \algorithmicfor}
\usepackage[preprint]{neurips_2025}

\usepackage[utf8]{inputenc} 
\usepackage[T1]{fontenc}    
\usepackage{hyperref}       
\usepackage{url}            
\usepackage{booktabs}       
\usepackage{amsfonts}       
\usepackage{nicefrac}       
\usepackage{microtype}      
\usepackage{xcolor}         

\usepackage{amsthm}
\usepackage{verbatim}
\usepackage{bbm} 
\usepackage{mathtools}
\newtheorem{theorem}{Theorem}
\newtheorem{definition}{Definition}
\newtheorem{remark}{Remark}
\newtheorem{corollary}{Corollary}

\title{Intrinsic and Extrinsic Organized Attention: Softmax Invariance and Network Sparsity}

\author{Oluwadamilola Fasina\textsuperscript{1,2}, Ruben V.C. Pohle\textsuperscript{2,3}, Pei-Chun Su\textsuperscript{1,2}, Ronald R. Coifman\textsuperscript{1,2,4} \\
\textsuperscript{1} Department of Applied and Computational Mathematics, Yale University \\ \textsuperscript{2} Graduate School of Arts and Sciences, Yale University \\ \textsuperscript{3} Medical Sciences Division, University of Oxford \\ \textsuperscript{4} Department of Mathematics, Yale University}   

\begin{document}

\maketitle

\begin{abstract}

We examine the intrinsic (within the attention head) and extrinsic (amongst the attention heads) structure of the self-attention mechanism in transformers. Theoretical evidence for invariance of the self-attention mechanism to softmax activation is obtained by appealing to paradifferential calculus, (and is supported by computational examples), which relies on the intrinsic organization of the attention heads. Furthermore, we use an existing methodology for hierarchical organization of tensors to examine network structure by constructing hierarchal partition trees with respect to the query, key, and head axes of network 3-tensors. Such an organization is consequential since it allows one to profitably execute common signal processing tasks on a geometry where the organized network 3-tensors exhibit regularity. We exemplify this qualitatively, by visualizing the hierarchical organization of the tree comprised of attention heads and the diffusion map embeddings, and quantitatively by investigating network sparsity with the expansion coefficients of individual attention heads and the entire network with respect to the bi and tri-haar bases (respectively) on the space of queries, keys, and heads of the network. To showcase the utility of our theoretical and methodological findings, we provide computational examples using vision and language transformers. The ramifications of these findings are two-fold: (1) a subsequent step in interpretability analysis is theoretically admitted, and can be exploited empirically for downstream interpretability tasks (2) one can use the network 3-tensor organization for empirical network applications such as model pruning (by virtue of network sparsity) and network architecture comparison.\footnote{Code is available at \href{https://github.com/obfasina/OrganizedAttention/}{https://github.com/obfasina/OrganizedAttention/}}

\end{abstract}

\section{Introduction}

Its widely accepted the transformer architecture \cite{vaswani2017attention,tay2022efficient,khan2022transformers,lin2022survey} transformed the fields of language and vision processing \cite{nadkarni2011natural,han2022survey}; however, there still exists a measurable gap between its theoretical underpinnings, and incidentally, the interpretability of the transformer architecture \cite{chefer2021transformer,rigotti2021attention,pan2021ia,yun2019transformers,perez2021attention}. Examples of efforts to close such gaps are \cite{ahn2023transformers,liu2022transformers} who discuss theoretical learning capabilities of transformers and \cite{zhao2024explainability} and \cite{rauker2023toward} who discuss the explainability and internal structure of transformers. Studies of this flavor have aided in closing the gap between theory, empirical findings, and explainability of the fabric of transformer architectures, and we aim to contribute to this by approaching these issues through the conceptual and mathematical lens of tensor organization.

The self-attention mechanism is used in virtually every transformer network and has been proven to be  crucial, in fact, necessary for transformers to learn effectively \cite{abnar2020quantifying,kreuzer2021rethinking}. Indeed, its ability to capture pairwise relationships amongst input tokens allows it to learn the inherent structure of the data in consideration. We, like others, conjecture that exploitation of information carried by the attention heads brings light to salient network structure. We constrain our focus to two types of analysis: intrinsic and extrinsic self-attention organization. We show theoretically - and supplement with computational examples - that if one produces an intrinsic organization of the attention head to have some regularity, the self-attention mechanism is invariant to the softmax function. Secondly, we describe an algorithm that admits an extrinsic organization of the attention heads; that is, an organization among all the heads in the network which allows us to probe the network structure, and consequently, carry out signal processing tasks economically.

\section{Paraproduct Decompositions and Tensor Organization}

We rely on two principal analytical tools for imputing salient intrinsic and extrinsic organization of attention heads: paraproduct decompostions and tensor organization. Paraproduct decompostions, or so called paradifferential calculus,  was initiated by \cite{bony1981calcul} and had a sizable impact in diverse subdisciplines of pure and applied mathematics such as operator theory, fourier analysis, and PDE. \cite{bahouri2011fourier,danchin2005fourier,hormander1990nash}. Informally, paradifferential calculus is concerned with decomposing operators into their low and high frequency constituents. Recently, Bony's work was extended by \cite{fasina2025quasilinearization}, facilitating similar decompositions for smooth nonlinear transformations of tensors with mixed-holder regularity. \\

The second tool is the questionnaire method originally described in \cite{coifman2011harmonic,gavish2012sampling,ankenman2014geometry} for organizing tensors. The questionnaire method relies on a key insight: understanding coordinates of each axis of a tensor as graphs to be organized with respect to the geometry of the other graphs induced by the axes of a tensor leads to a coherent tensor organization. Crucially, such a tensor organization allows for the data to be smooth with respect to the underlying geometry (the tensor product of graphs defined with respect to each axis in this setting), permitting one to consider further tensor inference questions such as sparsity and approximation. Indeed, in \cite{mishne2016hierarchical,georgiou2024clutter,sroczynski2025disorganized} the questionnaire method was utilized for inferring the spatio-temporal structure of neuronal data, enabling dynamic equation discovery governing chemical and bioengineering processes, and, general PDE discovery, respectively. As these two instruments are key to our empirical findings, we expound on their formulations in the subsequent subsections.

\subsection{Tensor Paraproducts}

We retain the notation used in \cite{fasina2025quasilinearization} to describe tensor paraproduct decompositions in our setting. First we begin with a few definitions describing Holder and mixed Holder regularity, a key regularity assumption required for the decomposition to hold.

\begin{definition} A function, $f: x \in [0,1] \to \mathbb{R}$ is $\alpha$-Holder with respect to the distance metric $d(x,y)$ satisfies the following the condition:

\begin{align}
& \frac{|f(x)-f(y)|}{|d(x,y)|^{\alpha}} \leq C \hspace{0.2 cm} \forall x,y \in \mathbb{R} && 
\end{align}
where $C,\alpha > 0$. We denote the set comprised of such functions as $\Lambda^{\alpha}([0,1])$
\end{definition}

\begin{definition} A function $f:(x,x') \in [0,1]^2 \to \mathbb{R}$ which is mixed $\alpha$-Holder with respect to the distance metrics $d(x,y),d'(x',y')$ satisfies the following conditions.

\begin{align}
& \frac{|f(x,x')-f(y,x')|}{|d(x,y)|^{\alpha}} \leq C, \frac{|f(x,x')-f(x,y')|}{ |d'(x',y')|^{\alpha}} \leq C, \nonumber \\
& \frac{|f(y,y') - f(x,y') - f(y,x') + f(x,x')|}{|d(x,y)|^{\alpha}| d'(x',y')|^{\alpha}} \leq C \hspace{0.1 cm} \forall x,x' \hspace{0.1 cm} \in [0,1]^2,
\end{align}

with $C, \alpha > 0$.  We denote this space as $\text{tensor-}\Lambda^{\alpha}([0,1]^2)$, but we refer to the space of such functions as $\Lambda^{\alpha}([0,1]^2)$ in the rest of the paper, for brevity.
\end{definition}

As described in \cite{gavish2010multiscale,meyer1990ondelettes}, Holder regularity and wavelets are intimately related as wavelet coefficients can be used to characterize Holder regularity and vice versa. Wavelets have a storied theoretical \cite{mallat1989theory} and applied \cite{coifman2021wavelets,gilbert2000multiresolution} history; this is due in no small part to their ease in computability (for certain wavelet families) and locality of in frequency and time. Here we consider Haar wavelet system though the tensor decomposition holds for any family of wavelets.

\begin{definition} (Wavelets and Tensor Wavelets for Haar System)
For $j,k \in \mathbb{N}$, the interval $I^j_k$ in direction $x$ is defined to be the set of points contained in the kth inteval at scale $j$. 

\begin{align}
I^j_k = \{x: x \in (2^{-j}k, 2^{-j}(k+1)] \}
\end{align}

We can then construct the scaling,  $\phi^j_k(x)$,  and wavelet functions, $\psi^j_k(x)$,  supported on the kth interval at scale j: 

\begin{align}
\phi^j_k(x)
\begin{cases}
    1 & x \in I^j_k \\
    0 & x \notin I^j_k
\end{cases}
\end{align}
\begin{align}
\psi^j_k(x)
\begin{cases}
    1 & x \in I^j_{k+} \\
    -1 & x \notin I^j_{k-}
\end{cases}
\end{align}

where $I^j_{k-} = \{ x :  x \in  ( \inf_I (I^j_k) ,  \inf_I (I^j_k) + \frac{(\inf_I (I^j_k) + \sup_I  (I^j_k)  )}{2} ]  \} $ and $I^j_{k+} = \{ x :  x \in  (\frac{(\inf_I (I^j_k) + \sup_I  (I^j_k)  )}{2} , \sup_I (I^j_k) ] \} $. The scaling and wavelet functions in the $x'$ direction are defined the same modulo a change of notation:

\begin{align}
I^{j'}_k = \{x': x' \in (2^{-{j'}}k, 2^{-{j'}}(k+1)] \}
\end{align}

\begin{align}
\phi^{j'}_k(x')
\begin{cases}
    1 & x' \in I^{j'}_k \\
    0 & x' \notin I^{j'}_k
\end{cases}
\end{align}
\begin{align}
\psi^{j'}_k(x')
\begin{cases}
    1 & x' \in I^{j'}_{k+} \\
    -1 & x' \notin I^{j'}_{k-}
\end{cases}
\end{align}

where $I^{j'}_{k+},I^{j'}_{k-}$ are defined in the same sense as $I^j_{k+},I^j_{k-}$. Finally we define the tensor scaling and tensor wavelet functions for $x,x' \in [0,1] \times [0,1]$

\begin{align}
\phi^{j,j'}_{k,k'}(x,x') =  \phi^j_k(x) \otimes \phi^{j'}_k(x')
\end{align}

\begin{align}
\psi^{j,j'}_{k,k'}(x,x') =  \psi^j_k(x) \otimes \psi^{j'}_k(x')
\end{align}

\end{definition}

\begin{remark}
The wavelet and tensor wavelet functions are defined for the sake of exposition but the decomposition is formulated and can be explicitly computed only with the tensor scaling function.
\end{remark}

\begin{theorem}

Suppose  $A \in C^2$, $ f \in \Lambda^\alpha([0,1]^2), 0 < \alpha < \frac{1}{2}$, then we obtain the following discrete multiscale tensor decomposition of $A(f)$:

\begin{align}
P^{j}P'^{j'}(f) = \sum_{k=1}^{2^j}  \sum_{k'=1'}^{2^{j'}} P^{j}_k P'^{j'}_{k'}(f) =  \sum_{k=1}^{2^j}  \sum_{k'=1'}^{2^{j'}} \big( \frac{1}{|R^{j,j'}_{k,k'}|}\int_{|R^{j,j'}_{k,k'}|} f(y,y') \phi^{j,j'}_{k,k'}(y,y') dy dy' \big) \chi_{R^{j,j'}_{k,k'}}(x,x')
\end{align}

\begin{align}
&  A(f) =  \sum_{j'=0'}^{N'} \sum_{j=0}^{N}  A'(P^jP'^{j'}(f)) [P^{j+1}P'^{j'+1}(f) - P^jP'^{j'+1}(f) - P^{j+1}P'^{j'}(f) + P^jP'^{j'}(f)] \nonumber \\
& + A''(P^jP'^{j'}(f))[P^{j+1}P'^{j'}(f) - P^jP'^{j'}(f)] [P^jP'^{j'+1}(f) - P^jP'^{j'}(f)]  + \Delta_{N,N'}(A,f)
\end{align}

\begin{align}
A(f) = \tilde{A}(f) + \Delta_{N,N'}(A,f)
\end{align}

\begin{align}
\Delta_{N,N'}(A,f) = A(f) - \tilde{A}(f) \in \Lambda^{2\alpha} ([0,1]^2)
\end{align}

where $\Delta_{N,N'}(A,f)$ is the residual between $A(f)$ and $\tilde {A}(f)$ up to scales $N,N'$. $P^jP'^{j'}$ is a convolution operators acting on $f$ with the tensor product of the scaling functions, $ \phi^{j,j'}_{k,k'}(x,x')$,  at scales $j,j'$ as its kernel. $R^{j,j'}_{k,k'} = I^j_k \times I^{j'}_{k'}$ is a dyadic rectangle constructed from taking the product of dyadic intervals $I^j_k, I^{j'}_{k'}$ at scales $j, j'$ such that $|R^{j,j'}_{k,k'} | = 2^{-(j+j')}$

\end{theorem}

\subsection{Tensor Organization}

\subsection{Questionnaire Algorithm}

The questionnaire algorithm organizes tensors in a multiscale sense by leveraging the hierarchal nature of partition trees. Informally, partition trees are a type of directed direct graph comprised of nodes with directed edges from parents to children. They can be binary where each parent has two children, or arbitrary - where no restriction is placed on the number of children a parent can have. In our setting, the nodes of the partition trees contain indices of a tensor axis (distinct from the sample associated with the index). As each axis has its own partition tree, in this situation we have three partition trees: a distinct tree for the row, column, and channel axes. We describe mathematically and algorithmically  the construction of such trees in the same sense as \cite{mishne2016hierarchical,ankenman2014geometry} in the context of transformer networks. \\

Let $\mathcal{H} = \{ 1,2, \ldots, n_H \}, \mathcal{Q} = \{1,2, \ldots, n_Q \}, \mathcal{K} = \{1,2, \ldots, n_K \}$, denote the set of indices for the head, query, and key axes. Define $\mathcal{T}_H, \mathcal{T}_Q, \mathcal{T}_K$, to be the partition trees defined on $\mathcal{H},\mathcal{Q},\mathcal{K}$, respectively which contain the nodes of the tree. Let $\mathcal{Q}_k^{l} \subset \mathcal{Q}$ be the kth node at level $l$ in the tree $\mathcal{T}_Q$ which contains a subset of the query indices. Analogously, we define $\mathcal{H}_k^l, \mathcal{K}_k^l$ - nodes at level $l$ and location $k$ in a tree. We use a type of EMD - proved to be equivalent to traditional formulations of EMD in \cite{leeb2015topics} and used  in \cite{ankenman2014geometry} - which computes affinities between 2D matrices based on where the coordinates of  the matrix lie in the cartesian product between the nodes of the trees defined on the set of indices of each axis. Suppose $ \mathbf{X} \in \mathbb{R}^{ |\mathcal{Q}| \times |\mathcal{K}| \times |\mathcal{H}|  }$ is the network 3-tensor we wish to process. Let $\mathbf{A}_p = \mathbf{X}[p,:,:], \mathbf{A}_j = \mathbf{X}[j,:,:]$. Then we define the EMD as the following:

\begin{align}
\text{EMD}_{\mathcal{T}_Q \times \mathcal{T}_K} (\mathbf{A}_p,\mathbf{A}_j) = \sum_{ \mathcal{Q}^l_k \in \mathcal{T}_\mathcal{Q}, \mathcal{K}^l_k \in \mathcal{T}_\mathcal{K}} m(\mathbf{A}_p - \mathbf{A}_j, \mathcal{Q}^l_k \times \mathcal{K}^l_k) w(\mathcal{Q}^l_k \times \mathcal{K}^l_k)
\end{align}

\begin{align}
m(\mathbf{A}_p,\mathcal{Q}^l_k \times \mathcal{K}^l_k) = \frac{1}{|\mathcal{Q}^l_k||\mathcal{K}^l_k|} \sum_{(q,k) \in  \mathcal{Q}^l_k \times \mathcal{K}^l_k } \mathbf{A}_p[q,k],  w(\mathcal{Q}^l_k \times \mathcal{K}^l_k) > 0
\end{align}

where $w(\mathcal{Q}^l_k \times \mathcal{K}^l_k) > 0$ is a weight value depending on the level and and size of the nodes in the trees and $(q,k)$ are the query and key indices contained in the cartesian product of a particualr set of nodes in the tree. Of course, the EMD with respect to the other axes are identical modulo a permutation of the partition trees incorporated into the computation. We summarize the questionnaire algorithm below:

\begin{algorithm}[H]
\label{alg:3dquest}
\caption{3D Questionnaire}
\begin{algorithmic}[1]
\Procedure{Bottom up flexible tree}{$\Psi$}
    \Require $\Psi = \{ \lambda_i \Psi_i(x) \}_{i=1}^n$ , an eigendecomposition of the Markov matrix
    \Ensure $\mathcal{T} = \{ 1,2, \ldots, n(\mathcal{T})\}$ where $n(\mathcal{T})$ is the number of nodes in partition tree
\EndProcedure 
\Procedure{Initialize query, key partition trees}{$\mathbf{X}$}
    \Require$ \mathbf{X} \in \mathbb{R}^{ |\mathcal{Q}| \times |\mathcal{K}| \times |\mathcal{H}|  }$
    \Ensure $\mathcal{T}^i_\mathcal{Q}$, $\mathcal{T}^i_\mathcal{K}$
    \State $\mathbf{X}_Q \in \mathbb{R}^{|\mathcal{Q}| \times |\mathcal{K}||\mathcal{H}|}$
    \State $\mathbf{G}_{\mathcal{Q}}[i,j] \gets \frac{<\mathbf{X}_Q^T(:,i),\mathbf{X}_Q^T(:,j)>}{  \lVert \mathbf{X}_Q^T(:,i) \rVert \lVert \mathbf{X}_Q^T(:,j) \rVert}$ $\forall i,j \in \{1,2, \ldots, |Q|\}$
    \State $\mathbf{D}[i,j] \gets \sum_m \mathbf{G}_{\mathcal{Q}}[i,m] \text{ for } i = j ,  \mathbf{D}[i,j] \gets 0 \text{ for } i \neq j$
    \State $\mathbf{M}_\mathcal{Q} \gets \mathbf{D}^{-\frac{1}{2}}\mathbf{D}^{-\frac{1}{2}} \mathbf{G}_{\mathcal{Q}} \mathbf{D}^{-\frac{1}{2}}\mathbf{D}^{\frac{1}{2}}, \phi : \mathbf{M}_\mathcal{Q} \to \{ \lambda_i \psi^{Q}_i(x) \}_{i=1}^n  = \Psi_Q$
    \State Repeat steps $4$ to $7$ for $\mathbf{X}_K \in \mathbb{R}^{|\mathcal{K}| \times |\mathcal{Q}||\mathcal{H}|}$ to obtain $\Psi_\mathcal{K}$
    \State $\mathcal{T}^i_\mathcal{Q} \gets $ \Call{Bottom up flexible tree}{$\Psi_{\mathcal{Q}}$}, $\mathcal{T}^i_\mathcal{K} \gets $ \Call{Bottom up flexible tree}{$\Psi_{\mathcal{K}}$}
    \State \Return $\mathcal{T}^{i}_\mathcal{Q}, \mathcal{T}^{i}_\mathcal{K}$
\EndProcedure
\Procedure{Tensor Organization}{$\mathcal{T}^i_\mathcal{Q}$,$\mathcal{T}_\mathcal{K}^i$ }
    \Require$ \mathcal{T}^i_\mathcal{Q},\mathcal{T}^i_\mathcal{K}, \mathbf{X} \in \mathbb{R}^{ |\mathcal{Q}| \times |\mathcal{K}| \times |\mathcal{H}|  }$
    \Ensure $\mathcal{T}^o_H, \mathcal{T}^o_Q, \mathcal{T}^o_K , \mathbf{G}_H, \mathbf{G}_K, \mathbf{G}_Q$
    \State $\mathcal{T}^i_\mathcal{Q}, \mathcal{T}^i_\mathcal{K} \gets$ \Call{Initialize query, key, partition trees}{$\mathbf{X}$}
     \For{$m = 1 \ldots$ $N$}
         \State $\mathbf{A}_p \gets \mathbf{X}[:,:,p],\mathbf{A}_j \gets \mathbf{X}[:,:,j]$, $\text{EMD}_{\mathcal{T}^i_\mathcal{Q} \times \mathcal{T}^i_\mathcal{K} }[\mathbf{A}_p,\mathbf{A}_j] \forall j,p \in \mathcal{H}$
         \State $\mathbf{E}_H \gets \text{EMD}_{\mathcal{T}^i_\mathcal{Q} \times \mathcal{T}^i_\mathcal{K} }$, $\mathbf{G}_H = e^{-\frac{\mathbf{E}_H}{\epsilon}}, \epsilon = \text{median}(\mathbf{E}_H) $, $\phi: \mathbf{G}_H \to \mathbf{M}_H \to \{ \lambda_i \psi_i(x) \}$
         \State $\mathcal{T}_H  \gets $ \Call{Bottom up flexible tree}{$\Psi_H$}
         \State Repeat steps 15-17 using key and head partition trees to generate query partition tree
         \State Repeat steps 15-17 using query and head partition trees to generate key partition tree
     \EndFor
    \State \Return $\mathcal{T}^o_H, \mathcal{T}^o_Q, \mathcal{T}^o_K , \mathbf{G}_H, \mathbf{G}_K, \mathbf{G}_Q$
\EndProcedure
\end{algorithmic}
\end{algorithm}

For further algorithmic details we refer the reader to \cite{ankenman2014geometry}, for more mathematical details, we refer the reader to \cite{gavish2012sampling}, and for a different but similar exposition in 3D, \cite{mishne2016hierarchical}. The questionnaire algorithm outputs the final affinities and partition trees along each axis which are useful for downstream analysis.

\subsection{Tri-Haar Basis and Sparsity}

Each of the nodes contained in the output trees, $\mathcal{T}^o_H,\mathcal{T}^o_Q,\mathcal{T}^o_K$, from \ref{alg:3dquest} contain the query, key, and head indices in each of the nodes. Consequentially, one has $f : \mathcal{T}^o_H \times \mathcal{T}^o_Q \times \mathcal{T}^o_K \to \mathbb{R} \in  \mathbf{X}, \mathbf{X} \in \mathbb{R}^{ |\mathcal{Q}| \times |\mathcal{K}| \times |\mathcal{H}| }$, meaning functions supported on the new geometry obtained from the questionnaire organization are just the network 3-tensor elements.  It follows that construction of a suitable basis to compress $f$ is natural and was first described in \cite{gavish2010multiscale,coifman2011harmonic}. We provide a brief description of a tri-haar basis on the obtained geometry following the notation of \cite{gavish2010multiscale,coifman2011harmonic}. \\

Define $\psi^{\mathcal{H}}_{l_0,k_0,j_0}(x),\psi^{\mathcal{Q}}_{l_1,k_1,j_1}(y),\psi^{\mathcal{K}}_{l_2,k_2,j_2}(z)$ to be the Haar wavelets associated with the nodes $\mathcal{H}_k^l\subset \mathcal{H},\mathcal{Q}_k^l \subset \mathcal{Q},\mathcal{K}_k^l \subset \mathcal{K}$ which contain subsets of head, query, and key indices, at node $k$ and level $l$ in the tree, where $j$ is the wavelet function associated with a particular node. Note $l_0 = 0_0, 0_1, \ldots, |n(l_0)| $ where $n(l_0)$ is the number of levels in the partition tree defined on the space of heads. The superscripts $k_0,j_0$ for the node location and the number of associated wavelet functions are indexed similarly, and a similar notation holds for $\{ l_1, k_1, j_1, l_2, k_2,j_2\}$, the superscripts for the query and key trees. Then one defines a tri-haar basis  $\psi^{\mathcal{H}}_{l_0,k_0,j_0}(x) \otimes \psi^{\mathcal{Q}}_{l_1,k_1,j_1}(y) \otimes \psi^{\mathcal{K}}_{l_2,k_2,j_2}(z)$ and its associated expansion coefficients: $ \bigcup_{i=0}^2 \alpha_{\{ l_i,k_i,j_i \}} = <f(x,y,z) (\psi^{\mathcal{H}}_{l_0,k_0,j_0}(x) \otimes \psi^{\mathcal{Q}}_{l_1,k_1,j_1}(y) \otimes \psi^{\mathcal{K}}_{l_2,k_2,j_2}(z) )>$. As in the Euclidean setting, such expansion coefficients can be used to measure function smoothness via the $l_p$ entropy of the expansion coefficients as described in \cite{coifman2011harmonic}. Of note, one can define 2D tensor bases with respect to the different basis vectors; in fact, the tensor basis comprised of the query and key bases will a central object in our analyses.

\section{Softmax Invariance and Sparsity}

We probe the intrinsic and extrinsic structure of the self-attention mechanism of two well-known transformer language and vision networks, respectively: Transformer-XL \cite{dai2019transformer} and Vit B16 \cite{dosovitskiy2020image}. For both networks, we provide numerical experiments to support our theoretical finding that self-attention is invariant to softmax activation under appropriate assumptions in section 4.1. We then use the questionnaire algorithm described in \ref{alg:3dquest} to organize the attention heads for both networks, permitting downstream network analysis. First, qualitative results of network organization are provided by a visualization of the attention heads in both networks using the diffusion maps \cite{coifman2006diffusion} algorithm, along with the resulting partition tree and adjacency matrix between the attention heads. Subsequently, we compute the expansion coefficients with respect to the query-key tensor basis and furthermore the $l_1$ entropy of each attention head of the networks to understand the network sparsity.\\

\subsection{Softmax Invariance}

\begin{theorem}

Let $f \in \mathbb{R}^{N \times N}$ be an attention head in a network and $\tilde{f} \in \mathbb{R}^{N \times N} \subset \Lambda^{\alpha}([0,1]^2)$ be a mixed $\alpha$-Holder approximation of $f$. Suppose $A_k(\tilde{f}_{i}) = \frac{e^{  \tilde{f}_{i,k}}}{\sum_j  e^{ \tilde{f}_{i,j}}}$. where $k,j$ are the column indices while $i$ is a row index such that $A_k(\tilde{f}_i)$ operates pointwise on the $i$th row and $k$th column of $\tilde{f} $ (i.e. the softmax activation function applied to $\tilde{f}$). Then the softmax activation on a mixed $\alpha$-Holder attention head can be decomposed into:

\begin{align}
&  A(\tilde{f}) =  A'(P^jP'^{j'}(f))[P^{j+1}P'^{j'+1}(f) - P^jP'^{j'+1}(f) - P^{j+1}P'^{j'}(f) + P^jP'^{j'}(f)] \nonumber \\
& +A''(P^jP'^{j'}(f))[P^{j+1}P'^{j'}(f) - P^jP'^{j'}(f)] [P^jP'^{j'+1}(f) - P^jP'^{j'}(f)]  + \Delta_{N,N'}(A,f) 
\end{align}
\end{theorem}

\begin{proof}
This is proved directly by appealing to \cite{fasina2025quasilinearization}.
\end{proof}

\begin{corollary}

If we assume that for a network $T:\{t_i\}_{i=1}^n \to \hat{v}_m$,  $L(v,\hat{v}_m) \to 0, \forall \hat{v}_m \in \mathcal{X}$, as $|[P^{j+1}P'^{j'+1}(f) - P^jP'^{j'+1}(f) - P^{j+1}P'^{j'}(f) + P^jP'^{j'}(f)]| \to \infty$ holds, then the softmax activation on a mixed $\alpha$-Holder attention head can be decomposed into:
\begin{align}
&  A(\tilde{f}) =  [P^{j+1}P'^{j'+1}(f) - P^jP'^{j'+1}(f) - P^{j+1}P'^{j'}(f) + P^jP'^{j'}(f)] \nonumber \\
& +[P^{j+1}P'^{j'}(f) - P^jP'^{j'}(f)] [P^jP'^{j'+1}(f) - P^jP'^{j'}(f)]  + \Delta_{N,N'}(A,f) 
\end{align}

where $L$ is a loss function,  $\{t_i\}_{i=1}^n $ is a set of  tokens input to the model , $\hat{v}_m$ is a prediction from $T$, $v$ is a ground truth, and $\mathcal{X}$ is a set containing possible predictions of  $T$ and the ground truth solution, and $P^{j}P'^{j'}(f)$ is the averaging operator defined in section 3. 

\end{corollary}

\begin{proof}
See Appendix
\end{proof}

\subsubsection{Network Descriptions}

The transformer-xl (TXL) model is comprised of 16 layers, each with a multi-head attention mechanism consisting of 8 heads. We train transformer-xl on the WikiText-103 \cite{merity2016pointer} dataset. Following training, we evaluate the network and 10 batches where each batch is comprised of 256 tokens. We then extract all the attention matrices for all 10 batches for downstream analysis. The VIT B16 model is comprised of 12 layers each with a multi-head attention mechanism comprised of 12 heads. The VIT model is trained on the CIFAR10 \cite{krizhevsky2009learning} dataset. Following training of the VIT model, we evaluate the network on 10 batches, where each batch is comprised of 197 tokens. Again, all the attention matrices for the 10 batches are used for downstream analysis. See the appendix for further experimental details.

\subsubsection{Softmax Invariance Computational Experiments}

\textbf{Nomenclature} We refer to a batch as a collection of tokens contained in a fixed context length. This terminology is used since we consider each token to be a sample.

For both the TXL and VIT models we considered 10 randomly selected batches of test data, which are projected onto the query and key matrices associated with each network, during inference. An attention matrix, $\mathbf{A}$, is constructed by projecting the data, $\mathbf{X}$ onto the query and key weight matrices, $\mathbf{W}_Q, \mathbf{W}_K$ : $\mathbf{A} = \text{softmax}(\frac{\mathbf{X} \mathbf{W}_Q \mathbf{W}_K \mathbf{X}^T}{\sqrt{D}})$ where $D$ is the number of tokens in each batch. The TXL network is comprised of 128 attention heads while the VIT network is comprised of 144. We can define the sets of attention matrices for each of the networks as:

\begin{align}
\mathcal{X}_{\text{TXL}} = \{ \{ \mathbf{A}^{(m)}_1 \}_{m=1}^{128}, \ldots, \{ \mathbf{A}^{(m)}_{10} \}_{m=1}^{128}  \}
\end{align}

\begin{align}
\mathcal{X}_{\text{VIT}} = \{ \{ \mathbf{A}^{(m)}_1 \}_{m=1}^{144}, \ldots, \{ \mathbf{A}^{(m)}_{10} \}_{m=1}^{144}  \}
\end{align}

where $\mathbf{A}_i^{(m)}$ is the attention head associated with the $m$th weight matrix for the $i$th batch. For the TXL network, $\mathbf{X} \in \mathbb{R}^{256 \times d}$ where $d$ is the model dimension while for the VIT network, $\mathbf{X} \in \mathbb{R}^{197 \times d}$ where $d$ is the model dimension for the VIT network. 256 and 197 correspond to the number of tokens in the TXL and VIT networks, respectively.\\

\begin{figure}[H]
    \centering
    \includegraphics[width=0.6\textwidth]{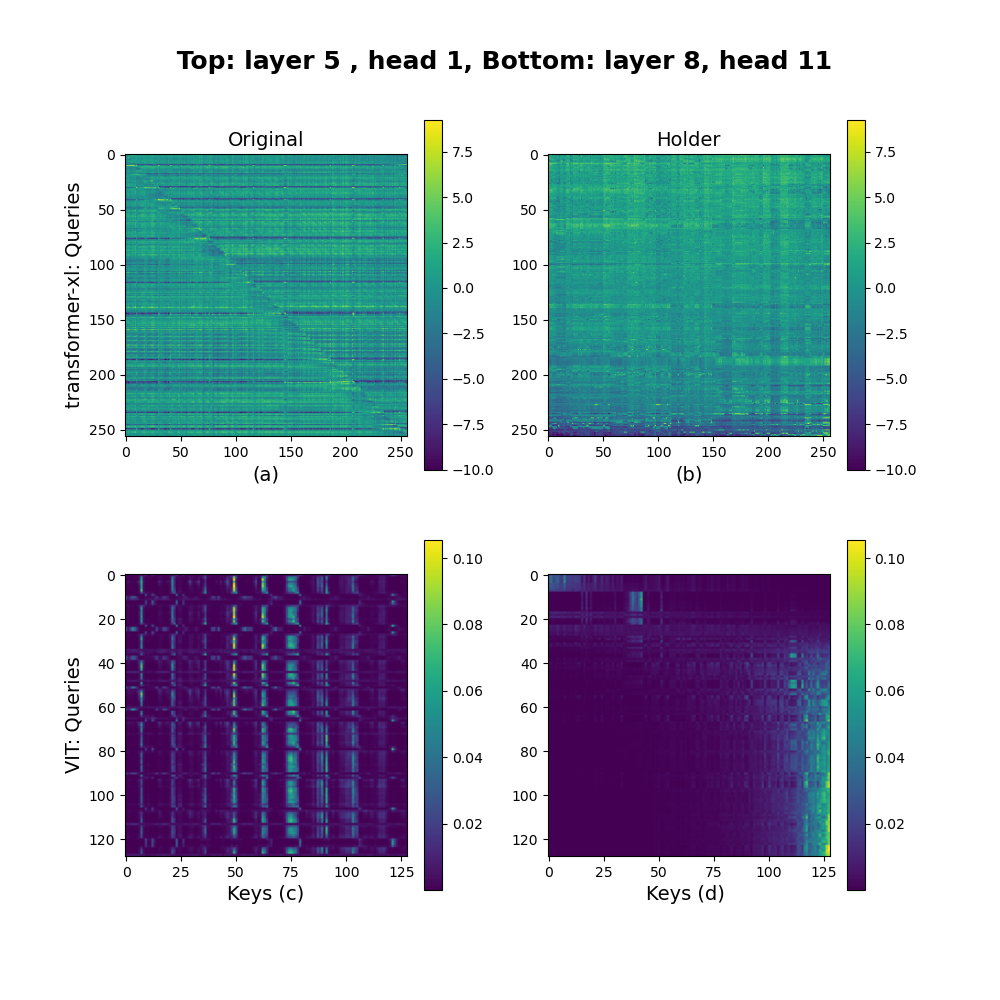} 
    \caption{Organization of two randomly selected attention heads (top is from TXL, bottom is from VIT) to have mixed $\alpha$-Holder regularity (a) original TXL attention head (b) organization of TXL attention head (c) original VIT attention head (d) organization of VIT attention head to be have mixed $\alpha$-Holder regularity}
    \label{fig:1}
\end{figure}

We randomly selected one of the 10 batches from both the WikiText-103 and CIFAR-10 datasets. On the left column is a visualization of the extracted attention heads for the TXL (top) and VIT (bottom) networks. On the right are the corresponding mixed $\alpha$-Holder organizations of the matrix. We remark that the attention head for the VIT network was cropped to have a dimension (128) on each axis, as having dimensions which are a power 2 are a necessary requirement for the implementation. The organization of a matrix to be mixed $\alpha$-holder is simply a 2D version of the 3D questionnaire algorithm described in the main text; further details can be found in the references associated with the 3D questionnaire algorithm. This organization is essential, as the decomposition theorem in the main text only holds for smooth nonlinear transformation of functions with mixed $\alpha$-Holder regularity. One can qualitatively see the smoothness in the right column in contrast to the left for both the TXL and VIT networks.

\begin{figure}[H]
    \centering
    \includegraphics[width=0.6\textwidth]{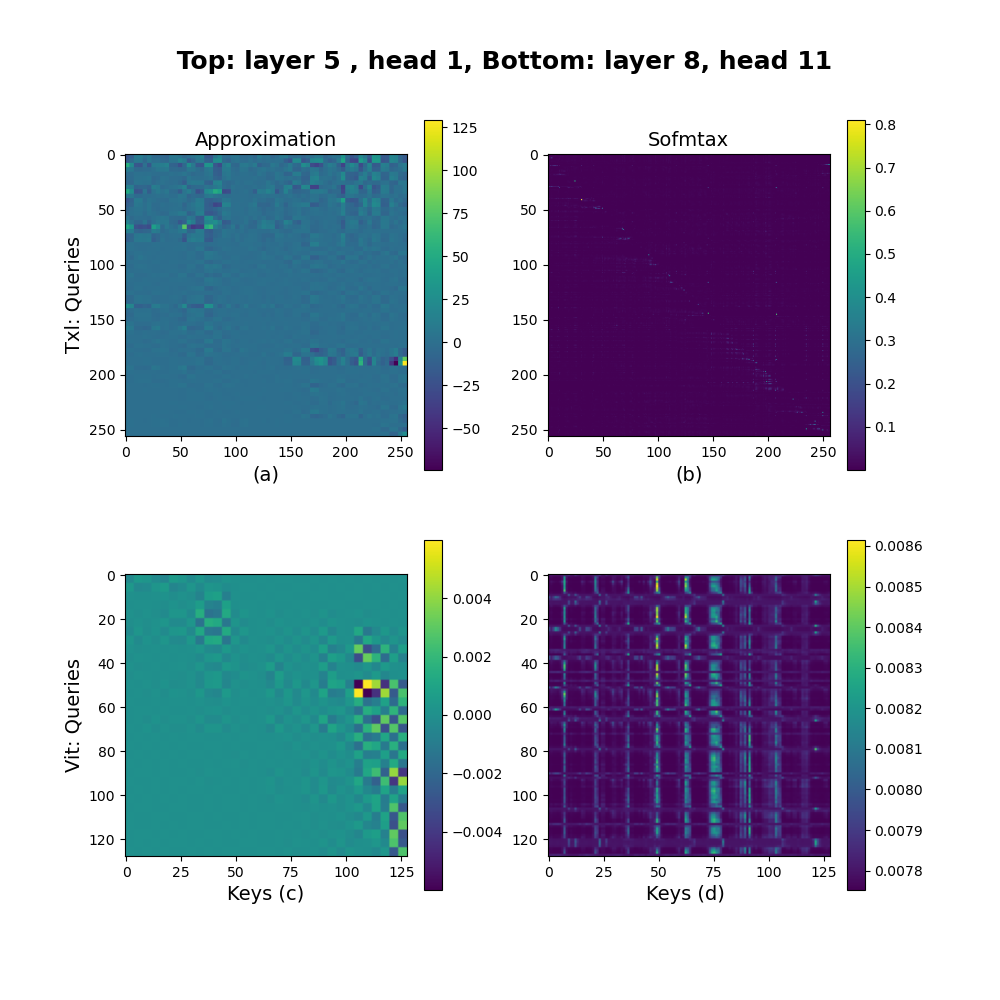} 
    \caption{Paraproduct decomposition and softmax attention mechanism applied to both attention heads in \ref{fig:1}. (a) approximation of TXL attention head with the paraproduct decomposition (b) Softmax of TXL attention head (c) approximation of VIT attention head with the paraproduct decomposition (d) softmax of VIT attention head}
    \label{fig:2}
\end{figure}

\vspace{-25mm}

We implement the decomposition theorem on the same attention heads considered in Figure 1. The left column is the approximation, the right column is the softmax activation function applied to the mixed $\alpha$-Holder attention heads. For the TXL attention, its readily seen the softmax activation function oversquashes certain coordiantes in the attention head, while the approximation preserves some of the structure. A similar phenomenon is seen for the VIT attention head, though less pronounced.

\vspace{-30mm}

\subsection{Network Sparsity}

We organized network 3-tensors of concatenated attention heads for both the TXL and VIT networks such that one can qualitatively and quantitatively assess the organization among all the heads in the network. For each batch comprised of 256 and 197 tokens for the WikiText-103 and CIFAR-10 datasets respectively, we deploy the questionnaire algorithm to organize the corresponding network 3-tensors: $ \{ \mathbf{A}_i^{(m)} \}_{m=1}^{128} \in \mathcal{X}_{\text{TXL}}, \{ \mathbf{A}_i^{(m)} \}_{m=1}^{144} \in \mathcal{X}_{\text{VIT}}$.

\begin{figure}[H]
    \centering
    \includegraphics[width=0.95\textwidth]{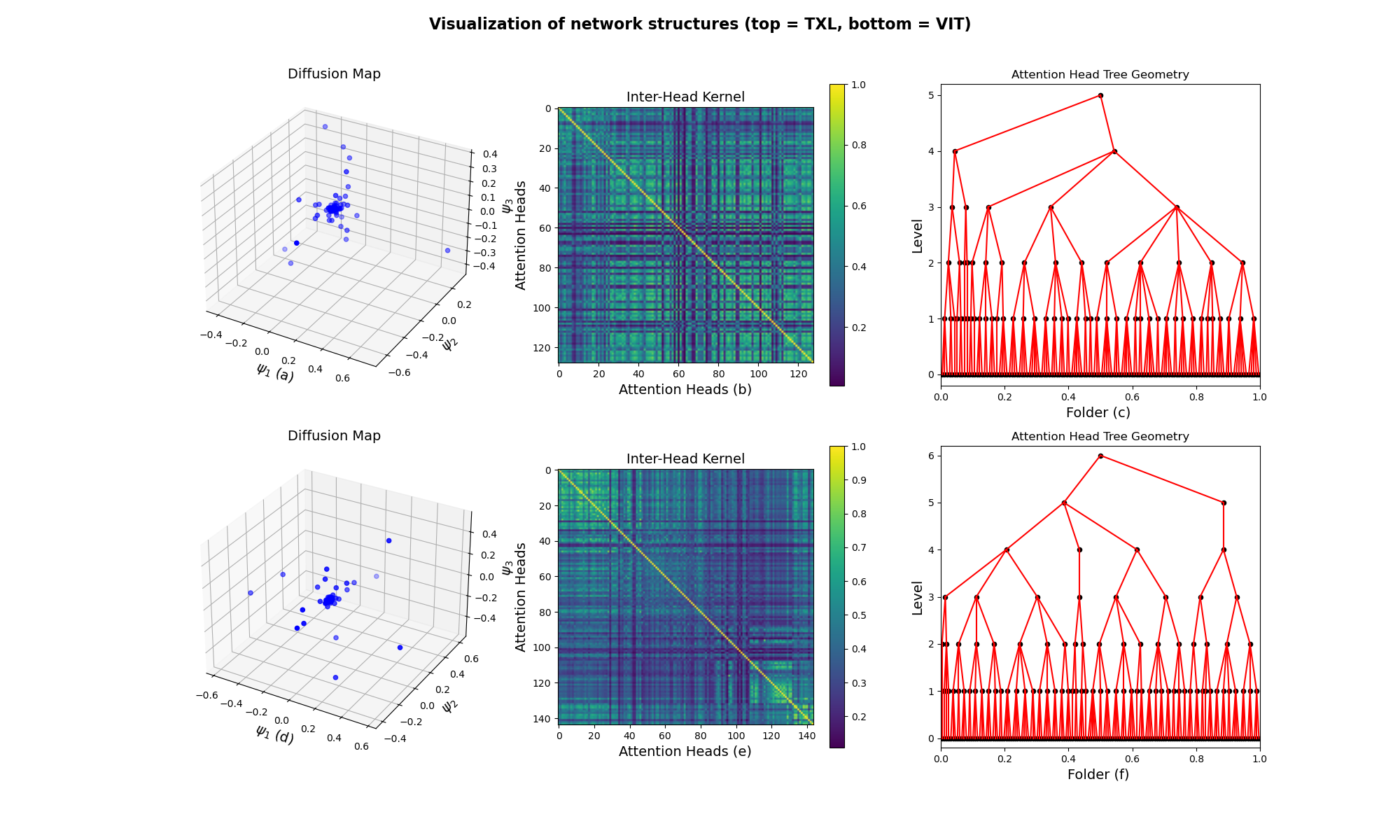} 
    \caption{Visualization of resulting network structure from diffusion maps algorithm. (a) Diffusion map embedding of TXL attention heads (b) Adjacency matrix of TXL attention heads (c) flexible tree comprised of TXL attention heads (d) Diffusion map embedding of VIT attention heads (e) Adjacency matrix of VIT attention heads (f) flexible tree comprised of VIT attention heads}
    \label{fig:3}
\end{figure}

\begin{figure}[H]
    \centering
    \includegraphics[width=0.95\textwidth]{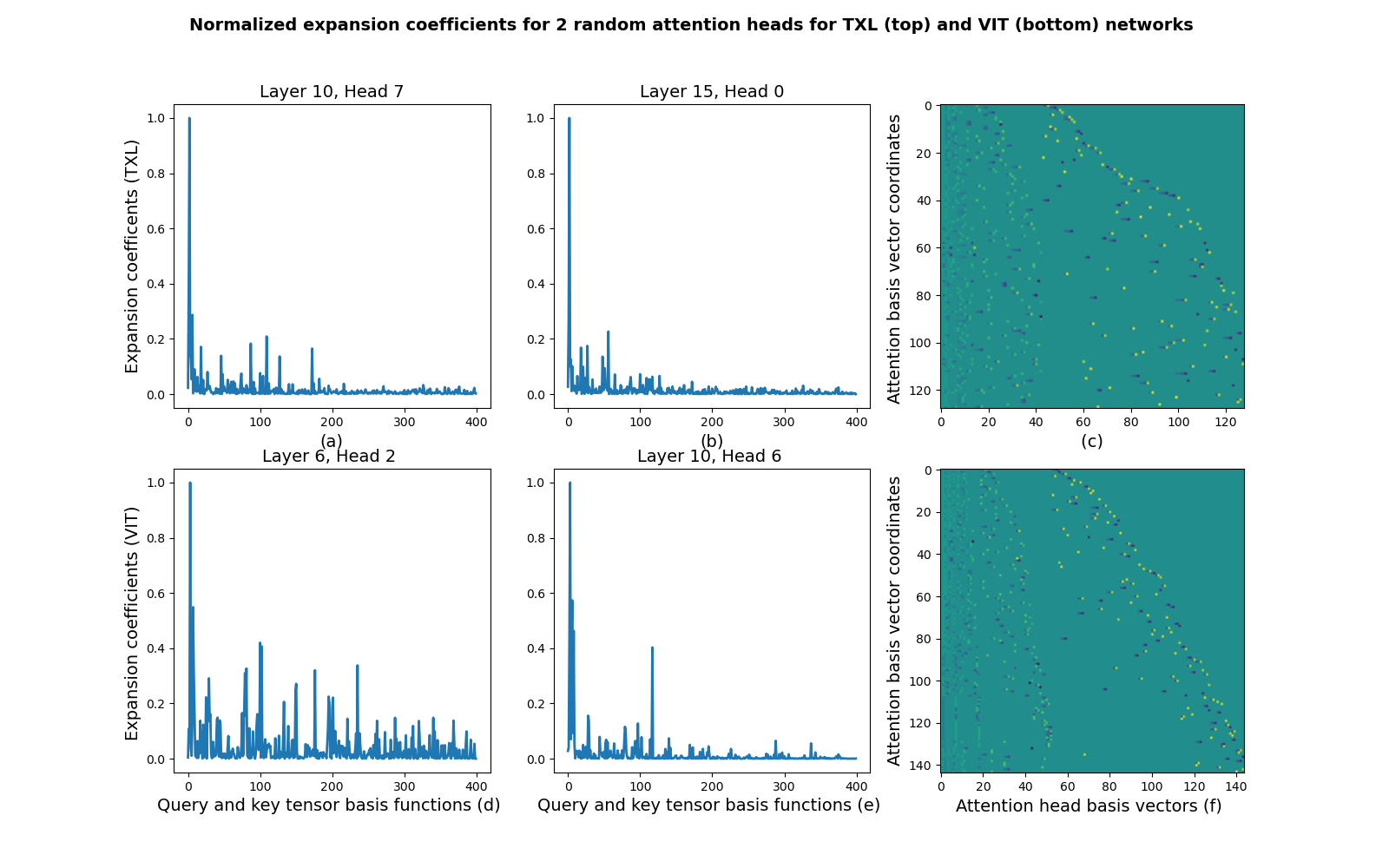} 
    \caption{Expansion coefficients of two random attention heads from both the TXL (top) and VIT (bottom) networks obtained from projecting an attention head into query-key tensor basis. For each network, plots of expansion coefficients are sorted in descending order from largest to smallest query-key tensor basis support size (a) Expansion coefficients for head 7 in layer 10 of the TXL network (b) Expansion coefficients for head 0, layer 15 in the TXL network (c) Haar basis vectors for the TXL attention head space (d)  Expansion coefficients for head 7 in layer 10 of the VIT network  (e)  Expansion coefficients for head 7 in layer 10 of the VIT network (f) Haar basis vectors for the VIT attention head space }
    \label{fig:4}
\end{figure}

The questionnaire algorithm outputs (a) a partition tree on the space of heads, where each tree contains a folder comprised of the head indices, (b) an affinity matrix comprised of pair-wise similarities between each of the attention heads in the network. We also subsequently built (c) a diffusion map using the affinity matrix of pair-wise relations between attention-heads. The top and bottom rows of Figure \ref{fig:3} are used to qualitatively assess the organization of the network for one batch of data for the TXL and VIT networks, respectively. The diffusion map on the TXL network exhibits some continuous structure, while the VIT network is more clustered and contains some outliers.  

In Figure  \ref{fig:4} we visualized the decay of the expansion coefficients by projecting each attention head on the the bi-haar basis of queries and keys. Let  $ \bigcup_{i=0}^1 \beta_{ \{l_i,k_i,j_i\}} = <f(x,y, \cdot),  \psi^{\mathcal{Q}}_{l_1,k_1,j_1}(x) \otimes \psi^{\mathcal{H}}_{l_2,k_2,j_2}(y)>$ be an expansion coefficient obtained from projecting into the bi-haar basis of queries and keys. The indices $(l,k,j)$ denote the level, location, and index of a basis function of a tree, respectively. the subscripts of $0$ and $1$ are used to denote the basis functions for the query and key binary trees, respectively.\\

The left and middle columns are plots of the expansion coefficients (sorted in descending order by the query and key bi-haar basis support size) obtained by  projecting two randomly selected heads into the query and key bi-haar basis. We consider the 400 tensor basis vectors with largest support size. The right column is a visualization of the basis vectors defined for the tree associated with the attention heads. The top and bottom rows are for the TXL and VIT networks, respectively. The distribution of expansion coefficients is similar for all attention heads.

\begin{figure}[H]
    \centering
    \includegraphics[width=0.95\textwidth]{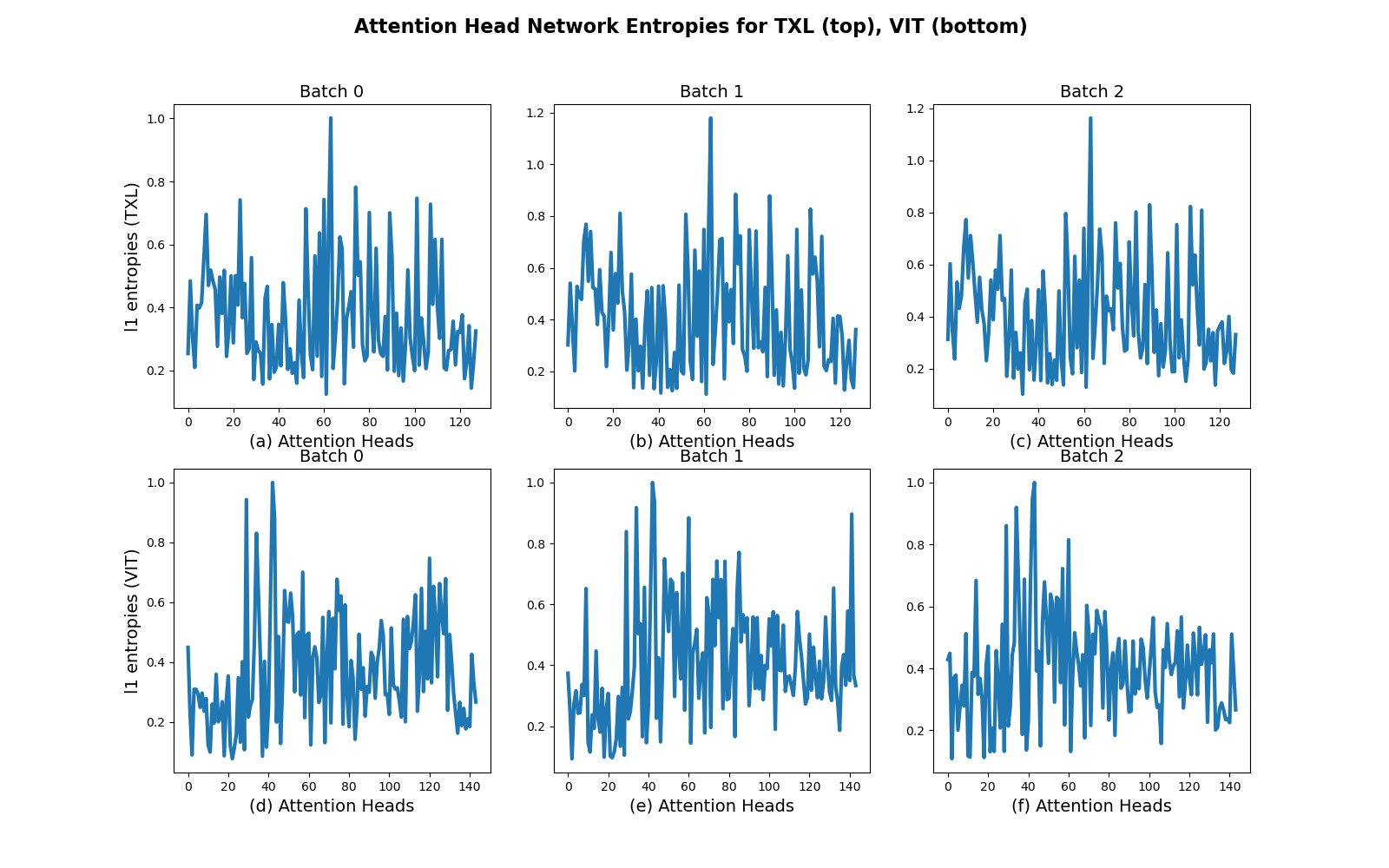} 
    \caption{Network $l_1$ entropies, obtained by projecting a function into tensor product of query and key bases. (a) - (c) Network entropy from projecting into query-key tensor basis (d) - (f) Network entropy from projecting into query-key tensor basis}
    \label{fig:sample}
\end{figure}

For each network we consider 3 batches each comprised of 256 and 197 tokens for the TXL and VIT, respectively. We compute the $l_1$ entropy of the expansion coefficients for each of the heads in the network by summing over the magnitude of the expansion coefficients associated with the top 400 basis vectors:

\begin{align}
\lVert \mathbf{\beta} \rVert_{l_1} =  \sum_{\text{supp} ( \psi^{\mathcal{Q}}(x) \otimes \psi^{\mathcal{H}}(y)) > \epsilon } | \bigcup_{i=0}^1\beta_{ \{l_i,k_i,j_i\}} |
\end{align}

where $\epsilon$ is a threshold for the 400 tensor basis vectors with largest support size, and $\mathbf{\beta}$ is a vector containing the 400 expansion coefficients. Note that each batch of data corresponds to a different collection of network attention heads since projection of this batch onto the query and key weight matrices results in a different attention head. Qualitatively, we can see a consistent spread of network entropies for both the TXL (top) and VIT (bottom) networks across the batches of data. An attention head with high $l_1$ entropy with respect to the query-key tensor basis relative to the rest of the network heads indicates that attention head is involved in a lot of processing as its attention head is picking up the sharp transitions within the attention head. The converse can be said for low $l_1$ entropy relative to the rest of the network. This qualitative finding (a) validates our organization method in the sense that a similar $l_1$  entropy network spread is seen for different batches within a network (b) suggests this technique could be used for model pruning; the head with low $l_1$ entropy do little processing and could be discarded.

\section{Conclusion and Limitations}

We demonstrated how penetrating the intrinsic and extrinsic structure leads to useful theoretical and empirical insights.  We rely on the questionnaire algorithm to intrinsically and extrinsically organize the attention heads.  The intrinsic organization permits the attention head to have Holder regularity, allowing us to connect this result to existing theory, revealing the self-attention mechanism is invariant to the softmax activation function. The extrinsic organization allows one to qualitatively and quantitatively understand  network structure. Qualitative inference constitutes construction of the diffusion maps algorithm on the affinity matrix between network heads as well as the hierarchical tree geometry of the attention heads, which are both output from the questionnaire algorithm. Quantitative inference consists of computing the entropy of the attention heads with respect to the bi-haar basis and the entropy of the network with respect to the tri-haar basis (see appendix). .  \\

Our main theoretical finding of softmax invariance has some restrictive assumptions. First, it requires the attention head to be organized to have mixed $\alpha$-Holder regularity. While this can be achieved using the questionnaire algorithm, the mixed $\alpha$-Holder attention head loses its original intrinsic structure since the matrix elements are permuted based on the location of the row and column indices in their respective binary trees. Incidentally, the mixed $\alpha$-Holder attention head (as organized by the questionnaire algorithm) and the original attention head are only equivalent up to large constants. A minor empirical restriction is implementation of the computational examples of the theorem requires the attention head to be a power of 2, which can force one to crop the attention matrix for further downstream analysis (e.g. VIT attention heads were cropped from 197 to 128). Additionally, the questionnaire algorithm can be sensitive to the algorithm parameters such as the affinity initialization and the weight function used during the iterative EMD process. Nevertheless, we envision and hope the theoretical and methodological contributions in this paper can be deployed for model pruning or general network interpretabilty tasks.

\bibliography{main}
\bibliographystyle{unsrt}

\newpage

\appendix

\section{Proofs}

\begin{corollary}

If we assume that for a network $T:\{t_i\}_{i=1}^n \to \hat{v}_m$,  $L(v,\hat{v}_m) \to 0, \forall \hat{v}_m \in \mathcal{X}$, as $|[P^{j+1}P'^{j'+1}(f) - P^jP'^{j'+1}(f) - P^{j+1}P'^{j'}(f) + P^jP'^{j'}(f)]| \to \infty$ holds, then the softmax activation on a mixed $\alpha$-Holder attention head can be decomposed into:
\begin{align}
&  A(\tilde{f}) =  [P^{j+1}P'^{j'+1}(f) - P^jP'^{j'+1}(f) - P^{j+1}P'^{j'}(f) + P^jP'^{j'}(f)] \nonumber \\
& +[P^{j+1}P'^{j'}(f) - P^jP'^{j'}(f)] [P^jP'^{j'+1}(f) - P^jP'^{j'}(f)]  + \Delta_{N,N'}(A,f) 
\label{eq:2}
\end{align}

where $L$ is a loss function,  $\{t_i\}_{i=1}^n $ is a set of  tokens input to the model , $\hat{v}_m$ is a prediction from $T$, $v$ is a ground truth, and $\mathcal{X}$ is a set containing possible predictions of  $T$ and the ground truth solution, and $P^{j}P'^{j'}(f)$ is the averaging operator defined in section 3. 

\end{corollary}

\begin{proof}

The intuition behind the corollary lies in the assumption: if the loss between the predicted and ground truth tokens decrease as the wavelet coefficients, $|[P^{j+1}P'^{j'+1}(f) - P^jP'^{j'+1}(f) - P^{j+1}P'^{j'}(f) + P^jP'^{j'}(f)]|$, of the self-attention matrix, $f$, tend to infinity, then it follows that the self-attention matrices with large wavelet coefficients helps the network learn.\\

Thus, if we only consider self-attention matrices, $\{ f_i \}_{i=1}^k$, with the property of  $|[P^{j+1}P'^{j'+1}(f) - P^jP'^{j'+1}(f) - P^{j+1}P'^{j'}(f) + P^jP'^{j'}(f)]| \to \infty$. We can show the collection of associated mixed $\alpha$-Holder functions $\{ \tilde{f}_i \}_{i=1}^k$ are invariant to the softmax activation. Let $A$ be the softmax activation function acting on some $\tilde{f}_m$ with the aforementioned property. Then  the constants, $A'(P^jP'^{j'}(\tilde{f}_m)) = A''(P^jP'^{j'}(\tilde{f}_m)) = e^{P^j P'^{j'}(\tilde{f}_m)} = e^{0} = 1$, since as $|[P^{j+1}P'^{j'+1}(\tilde{f}_m) - P^jP'^{j'+1}(\tilde{f}_m) - P^{j+1}P'^{j'}(\tilde{f}_m) + P^jP'^{j'}(\tilde{f}_m)]| \to \infty $,  $\mu(\tilde{f}_m) \to 0$, where $\mu(\tilde{f}_m)$ is the average of $\tilde{f}_m$. As the choice of $\tilde{f_m} \in  \{ \tilde{f}_i \}_{i=1}^k$ is arbitrary, the result holds for any attention-matrix with the specified property.

\end{proof}

\section{More experiments}

\subsection{Top/Bottom Attention Heads}

To illustrate how our methodology could be deployed for model pruning, we computed the $l_1$ entropy of each attention head in both networks with respect to the query-key tensor basis. We collected all network heads in the top 10 percent of $l_1$ entropies across 6 batches of data and did the same for the bottom 10 percent of $l_1$ entropies. We then extracted the top 3 and bottom 3 most commonly occurring attention heads in the top and bottom 10 percent of both networks and reported their modes on the right most column. The mode for the attention heads for the top 3 and bottom 3 attention heads appearing in the top and bottom 10 percent of network entropies is 6, meaning they appear in all batches considered for both networks. This suggests, the same attention heads do most of the processing regardless of the data seen, and perhaps more consequentially, the same attention heads do the fewest processing. The latter result implies one could remove attention heads with low $l_1$ entropy invariant to data from the network, potentially saving computation time.
 
\begin{table}[h!]
\centering
\caption{Top and Bottom $l_1$ Entropies of Attention Heads across 6 batches of data}
\begin{tabular}{cccc}
\toprule
Top 3/Bottom 3 (TXL) & (Layer, Head) & Count \\
\midrule
1   & (13,3)/(6,0)    & 6/6 \\
2  & (1,0)/(7,5)      & 6/6  \\
3   & (7,7)/(3,5)      & 6/6  \\
\midrule
Top 3/Bottom 3 (VIT) & (Layer, Head) & Count \\
\midrule
1   & (2,5)/(2,4)      & 6/6 \\
2   & (3,6)/(3,10)      & 6/6   \\
3   & (2,10)/(1,11)       & 6/6  \\
\bottomrule
\end{tabular}
\end{table}

\subsection{Tri-Haar Basis}

The structure of the entire network, and not just individual attention heads, can be understood through the lens of the tri-haar basis. We select 100 tri-haar basis vectors with the largest support size, and project each network 3-tensor into the tri-haar basis.\\

In table 2, we consider 5 batches of network tensor data from the WikiText-103 and CIFAR-10 datasets. The $l_1$ entropy is computed by summing the expansion coefficients of the largest 100 tri-haar basis vectors. Note there is no correlation between the two columns as each network is processing different batches. Of note, the $l_1$ entropies for TXL are consistently comparatively larger than VIT which can be interpreted as the TXL network processing more high frequency content than VIT due to the architectural set-up, but it is most likely related to the data being processed. \\

\begin{table}[H]
\centering
\caption{$l_1$ entropies of the TXL and VIT networks for 5 batches of WikiText-103 and CIFAR-10 data, respectively}
\begin{tabular}{cccc}
\toprule
Batch ID  & VIT & TXL \\
\midrule
1   & 6.84 & 1200.20 \\
2   & 3.30 & 2473.37 \\
3   & 2.47 & 2476.25 \\
4   & 5.47 & 1902.98 \\
5   & 9.77 & 2188.69 \\
\bottomrule
\label{entropies}
\end{tabular}
\end{table}

\begin{figure}[H]
    \centering
    \includegraphics[width=0.95\textwidth]{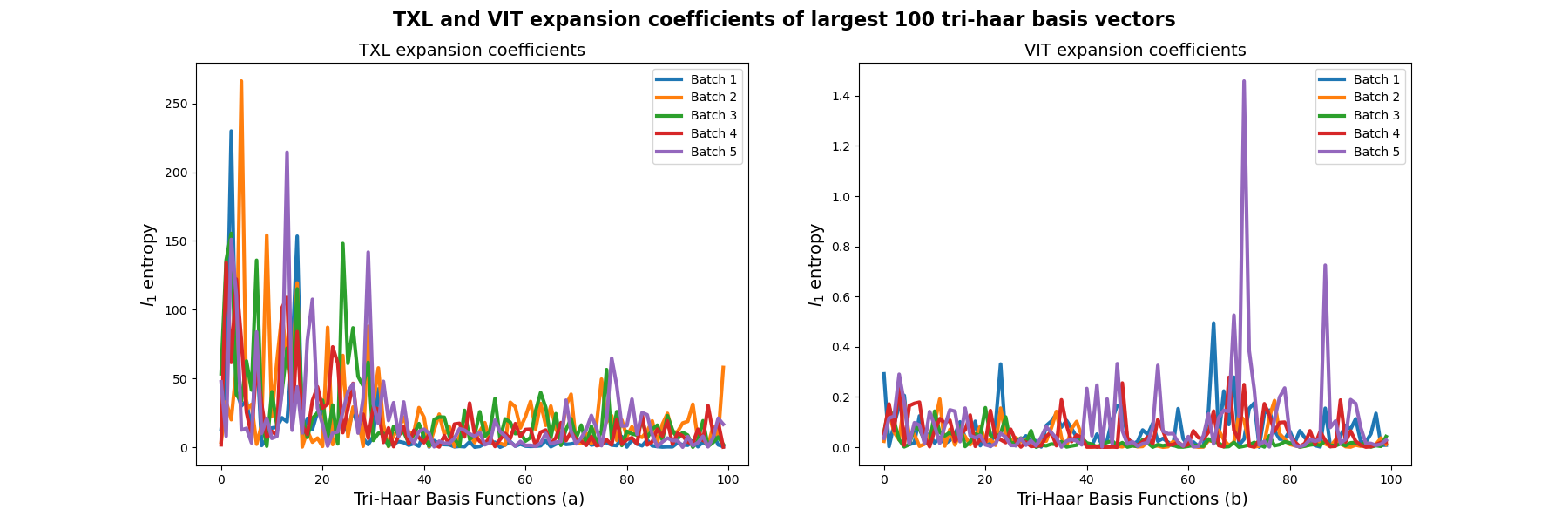} 
    \caption{TXL and VIT network entropies with respect to 100 tri-haar basis vectors with largest support size (a) VIT network entropies (b) TXL network entropies }
    \label{expan_coeffs}
\end{figure}

We visualized the expansion coefficients for the top 100 tri-haar basis vectors for the TXL and VIT networks for 5 batches associated with the WikiText-103 and CIFAR-10-10 datasets. The distribution of expansion coefficients is relatively consistent for both networks, with the TXL expansion coefficients being comparatively larger. Both table \ref{entropies} and figure  \ref{expan_coeffs} suggests the proposed methodology is suitable for comparative analysis amongst different networks and datasets.

\section{Experimental Details}

\subsection{Experimental results and reproducibility/settings and details/compute resources}

\textbf{TXL/WikiText-103}

We trained the transformer-xl (TXL) model \href{https://github.com/kimiyoung/transformer-xl}{TXL} on the \href{https://paperswithcode.com/dataset/wikitext-103}{WikiText-103} dataset. We used a model dimension (i.e. dimension of each token) of 410 with a context length of 256. We set the number of self-attention layers to 16, with each layer containing 8 attention heads. We trained with the adam optimizer with a learning rate of $2.5 \times 10^{-3}$. During test time, we randomly select all 128 network attention heads associated with 10 batches (collection of 256 tokens) of data. This model was trained with NVIDIA's A100 GPU.

\textbf{VIT/CIFAR-10}

We trained the vit b16 model \href{(https://docs.pytorch.org/vision/stable/models/generated/torchvision.models.vit_b_16.html)}{VIT} on the \href{https://www.cs.toronto.edu/~kriz/cifar.html}{CIFAR-10} dataset. We use pre-trained ViT B/16 weights and perform 1 epoch of fine-tuning training. We used the adam optimizer with a learning rate of $3 \times 10^{-5}$ and cross entropy loss. During test time, we randomly select all 144 network attention heads associated with 10 batches (collections of 197 tokens of data) This model was trained with NVIDIA's T4 GPU.

\subsection{Open access to data and code}

All datasets used are publicly available. The code can be found at \href{https://github.com/obfasina/OrganizedAttention/}{https://github.com/obfasina/OrganizedAttention/}

\subsection{Experimental statistical significance}

All computational experiments are done to demonstrate potential use cases; no general claims are made about concerning network inference. Thus, we do not consider the statistical significance of our results.

\end{document}